\documentclass[12pt]{article}
\usepackage{amsmath,amssymb,amsthm}
\usepackage{url}
\usepackage{graphicx}
\usepackage[usenames,dvipsnames]{xcolor}
\usepackage[ocgcolorlinks,colorlinks=true,urlcolor=blue,linkcolor=blue]{hyperref}
\usepackage{enumerate}
\usepackage{verbatim}
\usepackage{booktabs}
\usepackage[nodisplayskipstretch]{setspace}
\usepackage{multirow}
\usepackage{paralist}

\usepackage{upgreek}
\usepackage{tikz-cd}
\newif\ifKfour 
\Kfourfalse

\newtheorem{theorem}{Theorem}[section]
\newtheorem{lemma}[theorem]{Lemma}
\newtheorem{problem}[theorem]{Problem}

\newtheorem{corollary}[theorem]{Corollary}
\newtheorem{proposition}[theorem]{Proposition}

\theoremstyle{remark}
\newtheorem{rem}[theorem]{Remark}
\theoremstyle{definition}

\newcommand\R{\ensuremath{\mathbb{R}}}

\newcommand{\IH}{\it{IH}}
\newcommand{\conv}{\rm{conv}}

\DeclareMathOperator{\sd}{sd}

\DeclareMathOperator{\st}{st}

\DeclareMathOperator{\Pyr}{Pyr}

\title{QGLBT for polytopes}
\author{Karim Adiprasito\thanks{Einstein Institute of Mathematics,
The Hebrew University of Jerusalem, Jerusalem, 91904 Israel.
Partially supported by supported by ERC StG 716424 - CASe and ISF Grant 1050/16.},
Mikhail Burens\thanks{Einstein Institute of Mathematics,
The Hebrew University of Jerusalem, Jerusalem, 91904 Israel.},
and Eran Nevo\thanks{
Einstein Institute of Mathematics,
The Hebrew University of Jerusalem.
Partially supported by Israel Science Foundation grant ISF-1695/15, by grant 2528/16 of the ISF-NRF Singapore joint research program, and by ISF-BSF joint grant 2016288.}
}
\date{\today}

	\newcommand{\AR}{\mathcal{A}}
	\newcommand{\BR}{\mathcal{B}}

\begin{document}
\maketitle
\begin{abstract}
We extend the assertion of the Generalized Lower Bound Theorem (GLBT) to general polytopes under the assumption that their low dimensional skeleton is simplicial, with partial results for the general case.
We prove a quantitative version of the GLBT for general polytopes, and use it to give a topological necessary condition for polytopes to have vanishing toric $g_k$ entry. As another application of the QGLBT we prove a conjecture of Kalai on $g$-numbers for general polytopes approximating a smooth convex body.
\end{abstract}

\section{Introduction}
The well known $g$-theorem \cite{Billera-Lee:g, Stanley:g, McMullen:g-conj} characterizes the face numbers of \emph{simplicial} polytopes, and in particular
says that the $g$-numbers are nonnegative. The Generalized Lower Bound Theorem (GLBT) \cite{McMullenWalkup:GLBC-71, Murai-Nevo:GLBT} characterizes the simplicial polytopes attaining equality $g_k=0$, being exactly the \emph{$(k-1)$-stacked} polytopes.
For general poytopes $P$, the toric $g$-vector, introduced by Stanley \cite{Stanley:IH}
, is computed from the face poset of the polytope, and coincides with the $g$-vector in the simplicial case. Karu \cite{Karu:g} proved nonnegativity of the toric $g$-vector by showing it computes the dimension of the primitive cohomology of the associated Combinatorial Intersection Homology module ${\IH}(P)$ (introduced in \cite{BBFK,B-Lunts}) w.r.t. a Lefschetz element.
The following problem naturally arise:
\begin{problem}
Given $1\le k\le d/2$, for which $d$-polytopes $P$ does toric $g_k(P)=0$?
\end{problem}
We provide a few results in this direction.
A polytope is \emph{$k$-simplicial} if all its $k$-dimensional faces are simplices. The following generalizes the GLBT:
\begin{theorem}\label{thm:toricGLBT}
  Let $P$ be a $(2k-1)$-simplicial $d$-polytope with $g_k(P)=0$ ($k\le d/2$). Let $\Delta$ be the collection of geometric $d$-simplices whose $(k-1)$-skeleton is a subcomplex of the face complex of $P$. Then $\Delta$ is a triangulation of $P$.
\end{theorem}
This result is tight: note that the pyramid over $P$ satisfies $g(\Pyr(P))=g(P)$. Thus by taking a $(d-(2k-1))$-fold pyramid $Q=\Pyr(\Pyr(\cdots(\Pyr(P))\cdots)$ over any simplicial $(2k-1)$-polytope $P$, we have $g_k(Q)=0$, $Q$ is $(2k-2)$-simplicial, and the clique complex over the $(k-1)$-skeleton of $Q$ is typically a strict subset of $Q$.

The key point in the proof of Theorem~\ref{thm:toricGLBT} is noticing that, up to degree $k$, the module ${\IH}(P)$ is isomorphic to the corresponding Stanley-Riesner (quotient) ring, and then use either of the proofs of the GLBT, \cite{Murai-Nevo:GLBT} or \cite{Adiprasito:toric}, via crystallization or propagation resp.

The only $d$-polytope with $g_1(P)=0$ is the $d$-simplex.
The simplest open case then is to characterize $2$-simplicial $4$-polytopes $P$ with $g_2(P)=0$. One reduces to the prime case, namely when $P$ has no missing tetrahedra $T$, as otherwise $P=P_1\cup_T P_2$ and each $g_2(P_i)=0$, $i=1,2$. Paffenholz and Werner \cite{Paffenholz-Werner} constructed examples of such polytopes with $n$ vertices, for each $n\ge 13$.
We give a structural condition that all such polytopes must satisfy,
given by restricting the following theorem to $k=2$.
Let $P_W$ denote the polytopal complex whose faces are all faces of a polytope $P$ whose vertexset is contained in $W$, where $W$ is a subset of the vertexset $V(P)$ of $P$. 
Say a simplex $\sigma$ is a \emph{missing simplex} in $P$ if $\partial \sigma$ is a subcomplex of $\partial P$ and $\sigma$ is not a face of $P$.

\begin{theorem}\label{thm:induced}
Let $P$ be $d$-polytope with $g_k(P)=0$, $d\ge 2k\ge 2$, and let $W$ denote any subset of the vertices of $P$. Then the embedding of $P_W$ into the closure of $\partial P \setminus P_{V(P)\setminus W}$ induces the zero map on the
 $(k-1)$th rational homology groups.
Moreover, any missing $k$-simplex in $P$ is contained in a facet of~$P$.
\end{theorem}

Theorem~\ref{thm:induced} is inspired by, and extends, Kalai's result \cite{Kalai:Rig, KalaiCA} that when such $P$ is simplicial it has no missing $k$-simplices\footnote{The case $d=2k>4$ was not worked out by Kalai; later Nagel~\cite[Cor.4.8]{Nagel-empty} proved this case as well.}.
More generally, as we shall see,
the following extension of the quantitative GLBT (QGLBT) by Adiprasito \cite{Adiprasito:toric} to the toric case gives a lower bound on $g_k(P)$ in terms of topological Betti numbers of induced subcomplexes of $P$. Specifically,

\begin{theorem}\label{thm:QGLBT}
Let $P$ be a $d$-polytope, and $W$ any subset of its vertices $V=V(P)$. Let $k\le \frac{d}{2}$. Then the induced simplicial subcomplex $P_W$ satisfies \[{\binom{d+1}{k}}\alpha_{k-1}(P_W)\le g_k(P).\]
\end{theorem}

Here
$\alpha_{k-1}(P_W)$ denotes the dimension of the image of \[H_{k-1}(P_W)\rightarrow H_{k-1}(\mathrm{cl}(\partial P\setminus P_{V\setminus W})),\]
where we consider homology with real coefficients.

We use Theorem~\ref{thm:QGLBT} to extend the recent proof of Kalai's conjecture on simplicial polytopes approximating smooth convex bodies \cite{Adiprasito-Nevo-Samper:smooth} to the non-simplicial case.

\begin{corollary}\label{cor:smooth}
Let $K$ be a smooth convex body in $\R^d$, and $(P_n)$ a sequence of $d$-polytopes such that $P_n\rightarrow K$ in the Hausdorff metric. Then for any $1\le k\le d/2$,
$g_k(P_n)\rightarrow \infty$ as $n\rightarrow\infty$.

Further, if $K$ has a $C^2$ boundary, and a $d$-polytope $P$ is $\epsilon$-close to $K$ for some small enough $\epsilon>0$, then $g_k(P)=\Omega(\epsilon^{-\frac{d-1}{2}})$.
\end{corollary}

Outline: in Section~\ref{sec:IH} we recall the construction of $\IH(P)$ and prove Theorem~\ref{thm:toricGLBT}, in Section~\ref{sec:QGLBT} we prove the QGLBT, namely Theorem~\ref{thm:QGLBT}, and deduce from it Theorem~\ref{thm:induced} and Corollary~\ref{cor:smooth}.

\section{Intersection cohomology for general fans}\label{sec:IH}
In order to work in the context of general polytopes, we use the Barthel-Brasselet-Fieseler-Kaup \cite{BBFK} and Karu \cite{Karu:g} construction of the equivariant intersection cohomology sheaf, constructed inductively on the $i$th skeleton,
by iteratively applying the Lefschetz theorem to faces of dimension $i-1$. In particular, the equivariant sheaf $L(P)$ of a polytope $P$ is constructed as a subspace of
$L(\sd P)$, where $\sd P$ is the simplicial polytope whose boundary complex is the derived subdivision of the boundary complex of $P$.
The stalk over a proper face $\sigma$ of $P$ in $L(P)$ is a free module over the primitive elements with respect to
the operation of the
Lefschetz element induced on ${\IH}(\sigma)$, the  (non-equivariant) intersection cohomology of $\sigma$. It follows in particular that low-degree intersection cohomology depends more on the simplicial structure than higher degrees. Specifically,
for a geometric simplicial complex $\Delta$ in $\mathbb{R}^d$ denote by $\AR(\Delta)$ the quotient of the Stanley-Riesner  ring $\R[\Delta]$ of $\Delta$ over $\mathbb{R}$ by the ideal generated by the $d$ elements of degree $1$ corresponding to the embedding of the vertices of $\Delta$ in $\mathbb{R}^d$.

\begin{rem}
We adopt Karu's abuse of notation and do not adopt the natural grading for intersection cohomology arising from toric geometry (where we would naturally only have intersection cohomology in even degrees) and instead use degrees coming from the underlying model for the intersection ring, the Stanley-Reisner ring.
\end{rem}

\begin{proposition}\label{prop:{IH}&A}
If the $(2k-1)$-skeleton  $X_{\le 2k-1}$ of a geometric polyhedral complex $X$ is simplicial, then ${\IH}^i(X)\cong \AR^i(X_{\le 2k-1})$ for every $i\le 2k$.
\end{proposition}
This proposition allows us to prove
Theorem \ref{thm:toricGLBT} by following either of the proofs \cite{Murai-Nevo:GLBT} and \cite{Adiprasito:toric}.


\begin{proof}[Proof of Theorem~\ref{thm:toricGLBT}
]
Recall $\Delta$ denotes the simplicial complex  consisting of all subsets $\sigma$ of vertices of the $d$-polytope $P$ all whose subsets of size $\le k$ are simplices in $\partial P$. Denote by $\Delta_W$ the induced subcomplex of $\Delta$ on the vertex set $W$, and by $V(F)$ the vertices of a face $F$ of $P$.
We show the following three properties for any $i$-face $F$ of $P$:

\begin{compactenum}[(A)]
\item $\Delta_{V(F)}$ is Cohen-Macaulay of dimension $\dim F$. \\
We can assume that $F$ is of dimension at least $2k$. It follows immediately from flabbiness of the intersection cohomology sheaf that $g_j(F)\le g_j(P)$ for all $j$~\cite{Braden-MacPherson, Braden-Remarks}; hence the claim follows at once from \cite[Cor. 4.7]{Adiprasito:toric} for $\Delta_{V(F)}$, which is indeed applicable using Proposition~\ref{prop:{IH}&A} and Karu's hard Lefschetz for polytopes~\cite{Karu:g}.

\item $\Delta_{V(F)}$ is a geometric complex embedded
 in $F$. \\
This is a result of McMullen, namely the argument in the proof of~\cite[Thm.4.1]{TSPMcMullen}; see also~\cite[Prop.3.4]{BD} and ~\cite[Lem.4.2]{Murai-Nevo:GLBT}.

\item Finally, conclude that $\Delta_{V(F)}$ triangulates $F$.\\
For this we need that the geometric realization of $\Delta_{V (F)}$ contains the boundary
of $F$, which we know by the induction hypothesis (for $i=2k$ this is the data
that $P$ is $(2k-1)$-simplicial). But every Cohen-Macaulay subcomplex of $\R^i$ and of dimension $i$ is a ball.

For $i=d$ we obtain that $\Delta$ is a geometric triangulation of $P$, as desired. \qedhere
\end{compactenum}
\end{proof}

\section{Quantitative generalized lower bound theorem and applications}\label{sec:QGLBT}

We shall predominantly need a notion of topology of induced subcomplexes. Let $X$ be a strongly regular CW complex, namely the open cells in $X$ are embedded and the intersection of closures of any two cells is the closure of a cell in the boundary of both.
Let $W$ be a subset of the vertexset of $X$. We denote by $X_W$ the collection of those faces whose vertices are subsets of $W$. Then $\alpha_{k-1}(X_W)$ denotes the dimension of the image of $H_{k-1}(X_W)$ in $H_{k-1}(\mathrm{cl}(X\setminus X_{V(X)\setminus W}))$.

\subsection{Proof of Theorem~\ref{thm:QGLBT}}
With the structure of intersection cohomology given in Section~\ref{sec:IH}, we conclude the proof of Theorem \ref{thm:QGLBT}. While it is possible to prove the theorem in the same way as in \cite{Adiprasito:toric}, this is a little cumbersome as the "support" of a Chow cohomology class is a little tricky to phrase in the intersection ring. Instead, we give a proof that
focuses on an argument similar to \cite{Kalai:Rig} and the appendix of \cite{Adiprasito-Nevo-Samper:smooth}, where we proved the same under the assumption that lower-dimensional cohomologies vanish.

\newcommand{\bigslant}[2]{{\raisebox{.3em}{$#1$} \Big/ \raisebox{-.3em}{$#2$}}}

\newcommand*\longhookrightarrow{\ensuremath{
\lhook\joinrel\relbar\joinrel\rightarrow}}
\newcommand*\longhookleftarrow{\ensuremath{
\leftarrow\joinrel\relbar\joinrel\rhook}}

Let $Q$ denote a simplicial polytope and let $W$ denote a subset of vertices of $Q$, which we may identify with a set of prime divisors. Assume that the closed neighborhood $N$ of the induced subcomplex $Q_W$, namely the subcomplex consisting of all faces of $Q$ that are contained in a face containing a vertex from $W$, is a regular neighborhood of $Q_W$.

Define \[\mathcal{I}\ :=\ 
\ker[  \AR(N)\rightarrow  \bigoplus_{w\in W} \AR(\st_w N)].\]
Set $\BR(N):=\AR(N)/\mathcal{I}$.
Then $\AR(N)\twoheadrightarrow \BR(N)$.

Then, by definition, we have an injection
\[\BR(N)\ \longhookrightarrow\ \bigoplus_{w\in W} \AR(\st_w N)\]
where $\AR(\st_w N)$ denotes the quotient ring of $\AR(N)$ corresponding to the closed star of $w$ in $N$.

Consider $\ell$ the class of an ample divisor in $\AR(N)$.
Then we have a diagram
\[\begin{tikzcd}
0 \arrow[]{r} & \BR^{k-1}(N) \arrow[]{r}\arrow[]{d}{\ell} & \bigoplus_{w\in W} \AR^{k-1}(\st_w N) \arrow[]{d}{\ell} \\
 0 \arrow[]{r} & \BR^{k}(N) \arrow[]{r}& \bigoplus_{w\in W} \AR^{k}(\st_w N)
 \end{tikzcd}
\]
where the second vertical map $\ell$, and therefore also the first, is an injection.

 Finally, we have an isomorphism
\begin{equation}\label{eq:hom}
H^{k-1}(N)^{\binom{d}{k}}\ \cong\ \ker [{\AR}^{k}(N) \longrightarrow \bigoplus_{w\in W} \AR^{k}(\st_w N)],
\end{equation}
 see \cite[Thm.2.2]{Novik-Swartz:socle}. Hence, we obtain an injection of $(H^{k-1}(N))^{\binom{d}{k}}$ into \[\bigslant{{\AR}^{k}(N)}{\ell{{\AR}^{k-1}(N)}}.\]
 
 \begin{rem}
More elementary, one can prove this fact as follows.
Consider the chain complex $\widetilde{\mathcal{P}}^\bullet$ defined as
		\[0\ \longrightarrow\ \R^\ast[\Delta]\ \longrightarrow\ \bigoplus_{v\in \varDelta^{(0)}} \R^\ast[\st_v \Delta]\ \longrightarrow\ \cdots\ \longrightarrow\  \bigoplus_{F \in \varDelta^{(d-1)}} 
		\R^\ast[\st_F \Delta] \  \longrightarrow\ 0,\]
and tensor it with the Koszul complex ${K}^\bullet$ given by the linear system of parameters. Computing the second page of the associated filtrations, we obtain Isomorphism~\eqref{eq:hom}. If $d>2k$, this can be strengthened to
\[H^{k-1}(N)^{\binom{d+1}{k}}\ \cong\ \ker\left[\bigslant{{\AR}^{k}(N)}{\ell{{\AR}^{k-1}(N)}}
\rightarrow \bigoplus_{w\in W} \bigslant{{\AR}^{k}(\st_w N)}{\ell{{\AR}^{k-1}(\st_w N)}}
\right].
\]
If $d=2k$, then we only obtain an injection
\[H^{k-1}(N)^{\binom{d+1}{k}}\ \longhookrightarrow\ \bigslant{{\AR}^{k}(N)}{\ell{{\AR}^{k-1}(N)}}.
\]
These maps are realized explicitly by the Ishida complex~\cite[Secs.4,5]{Ishida} and \cite{Oda}.
 \end{rem}

In particular,
\begin{lemma}
Under the above conditions, $H^{k-1}(N)^{\binom{d+1}{k}}$ injects into $\bigslant{{\AR}^{k}(N)}{\ell{{\AR}^{k-1}(N)}}$.
\end{lemma}

We now go back to $P$. Subdivide it barycentrically twice to obtain $P'$. Then the corresponding subdivision $P'_{W'}$ of $P_W$ is an induced subcomplex satisfying the regular neighborhood condition for the previous lemma. Let $N$ denote its closed neighborhood. Therefore, $H^{k-1}(N)^{\binom{d+1}{k}}$ injects into $\bigslant{{\AR}^{k}(N)}{\ell{\overline{\AR}^{k-1}(N)}}$.

Now, using the decomposition theorem, the cokernel of the pullback inclusion of $\IH(P)$ to $\AR(P')$ is generated by the images of the Gysin maps (see \cite[Section~6.5]{FultonIT}).
Thus, these images
correspond to cohomologically trivial cycles in $H^{k-1}(\mathrm{cl}(\partial P\setminus P_{V\setminus W}))$. Hence, we conclude an injection of the image of $H_{k-1}(N)^{\binom{d+1}{k}}\cong H_{k-1}(P_W)^{\binom{d+1}{k}}$ in $H_{k-1}(\mathrm{cl}(\partial P\setminus P_{V\setminus W}))^{\binom{d+1}{k}}$ into
\[\bigslant{\IH^{k}(P)}{\ell {\IH^{k-1}(P)}. \ \ \ \qed}\]

%

\subsection{Applications: proofs of Theorem~\ref{thm:induced} and Corollary~\ref{cor:smooth}.}
\begin{proof}[Proof of Theorem~\ref{thm:induced}]
That the image of $H_{k-1}(P_W)$ in $H_{k-1}(\mathrm{cl}(\partial P \setminus P_{V\setminus W}))$ is zero is immediate from the QGLBT, Theorem~\ref{thm:QGLBT}. 

Finally, it remains to argue that a missing simplex $\partial \sigma$ of dimension $k$ in $P$ is indeed contained in a face of $P$. To this end, let $v$ denote one of its vertices, and observe that following Kalai \cite{Kalai:Rig}, the face $\sigma-v$ must be contained in the star of $v$ in $\partial P$; indeed, if that is not the case, \[\bigslant{\IH^{k}(\st_v P \cup \sigma-v)}{\ell {\IH^{k-1}(\st_v P \cup \sigma-v)}}\] is nontrivial.
\end{proof}

\begin{proof}[Proof of Corollary~\ref{cor:smooth}]
When the convex body $K$ is smooth, let $b$ be bigger then any given constant, and when $K$ has a $C^2$ boundary, let $b=\Omega(\epsilon^{-\frac{d-1}{2}})$ when $P$ is $\epsilon$-close to $K$ for small enough $\epsilon >0$.
With Theorem~\ref{thm:QGLBT} (QGLBT) at hand, we proceed exactly as in \cite{Adiprasito-Nevo-Samper:smooth} and find
$b$ copies $\gamma_1,\ldots,\gamma_b$ of the $(k-1)$-sphere in $\partial K$, and $0<\epsilon'<\epsilon$ (where $P$ is $\epsilon$-close to $K$, for some $\epsilon$ small enough) such that:

(i) the $\epsilon$-neighborhoods $\gamma_i+\epsilon$, $1\le i\le b$, are pairwise disjoint in $\partial K$, and each $\gamma_i+\epsilon$ deformation retracts to $\gamma_i$;

(ii) if 
$v,u\in V(P)$ such that $v\in \gamma_i+\epsilon'$ and $u\in \gamma_j+\epsilon'$ for some $i\neq j$, then no proper face of $P$ contains both $v$ and $u$;

(iii) for $W_i:=V(P)\cap \gamma_i+\epsilon'$, the following inclusions hold
\[\pi_P\gamma_i \subset P_{W_i} \subseteq \partial P\cap {\conv} W_i \subseteq \gamma_i+\epsilon\]
where $\pi_P$ denotes the closest point projection to $P$.

We conclude the proof by noticing that (i) and (iii) imply $\alpha_{k-1}(P_{W_i})\ge 1$, and thus, by (ii), for $W=\uplus_i W_i$, $\alpha_{k-1}(P_W)=\sum_{i=1}^{b}\alpha_{k-1}(P_{W_i})\ge b$.
\end{proof}

\bibliographystyle{myamsalpha}
\bibliography{bib_g_k}

\end{document}